\documentclass[12pt]{article}
{\begin{Sbox}\begin{minipage}}%
{\end{minipage}\end{Sbox}\fbox{\TheSbox}}%

\usepackage[psamsfonts]{amssymb}
\usepackage{amsmath, amsthm, anysize, enumerate, color, url}
\usepackage{graphicx}
\usepackage{epstopdf}
\usepackage{epsfig}
\usepackage{subfigure}

\usepackage{algorithm}
\usepackage{algpseudocode}

\usepackage{empheq}

\usepackage{natbib}
\usepackage{threeparttable}

\newtheorem{theorem}{Theorem} 
\newtheorem{lemma}{Lemma}
\newtheorem{proposition}{Proposition} 

\usepackage[colorlinks, breaklinks,
                   linkcolor=red,
                  anchorcolor=blue,
                  citecolor=blue
                  ]{hyperref}

\usepackage{breakurl}

\newtheorem{definition}{Definition}[section]

\theoremstyle{definition}
\newtheorem{remark}{Remark}  

\newcommand{\E}{\mathbb{E}}

\newcommand{\R}{\mathbb{R}}

\renewcommand{\P}{\mathbb{P}}

\newcommand{\bs}{\boldsymbol}

\usepackage{setspace}

\begin{document}

\title{Weighted directed networks with a differentially private bi-degree sequence}

\author{
Qiuping Wang\thanks{Corresponding author: Qiuping Wang, Department of Statistics, Central China Normal University, Wuhan, 430079, China.
Wang and Zhang contribute equally to this article. \texttt{Emails:}$^*$qp.wang@mails.ccnu.edu.cn,
$^\S$zhangxiao$\_$2019@outlook.com,
$^\ddag$ouyangyang$\_$2019@outlook.com,
$^\sharp$wangqian@mails.ccnu.edu.cn.
} 
\hspace{2mm}
Xiao Zhang$^\S$
\hspace{2mm}
Jing Luo\thanks{Corresponding author: Jing Luo, School of Mathematics and Statistics, South-Central University for Nationalities, Wuhan, 430079, China.
\texttt{Emails:} jingluo@mail.scue.edu.cn.}
\hspace{2mm}
Yang Ouyang$^\ddag$
\hspace{2mm}
Qian Wang$^\sharp$
\\~~\\
Central China Normal University$^{*,\ddag,\S,\sharp}$\\
South-Central University for Nationalities$^\dag$
}

\date{\empty}
\maketitle

\begin{abstract}

The $p_0$ model is an exponential random graph model for directed networks with the bi-degree
sequence as the exclusively sufficient statistic. It captures the network feature of degree heterogeneity.
The consistency and asymptotic normality of a differentially private estimator of
the parameter in the private $p_0$ model has been established. However, the $p_0$ model only focuses on binary edges.
In many realistic networks, edges could be weighted, taking a set of finite discrete values.
In this paper, we further show that
the moment estimators of the parameters based on the differentially private bi-degree sequence in the weighted $p_0$ model are consistent and asymptotically normal. Numerical studies demonstrate our theoretical findings.


\end{abstract}
\vskip 5pt

\section{Introduction}
Networks provide a convenience way for representing relationships between a set of individuals such as friendships of
social networks [e.g., \cite{amaral2000classes}], co-authors in collaboration networks [e.g., \cite{Newman2001}] and protein--protein interactions in biological networks [e.g., \cite{Han2005Effect}].
With the rapid development of information technology, more and more network data have been collected and stored.
As results, there has been a growing interest to analyze network data in statistics.
Many statistical models have been established to reveal essence from network data.
For instance,  \cite{holland1981exponential} proposed the $p_1$ model with the dyad independent assumption for modeling the variety of degrees and reciprocity in
binary directed graph data. \cite{frank1986markov} introduced the notion
of Markov dependence for graphs, which specifies that two
possible edges are dependent whenever they share a vertex conditional on all other edges, and established the Markov random graph models, in which the counts of $k$-stars and triangles are sufficient statistics in exponential-family distributions on graphs.
A more general form $p^*$ model was later given by \cite{Wasserman1996}, which  makes an extension from the single network to the multiple networks.
\cite{wang1987stochastic} introduced the stochastic block models explaining the block structure.
The asymptotic theories in these models have also been established [e.g., \cite{chatterjee2011random}, \cite{Yan2016asymptotics}, \cite{yan2013central}, \cite{ChatterjeeDiaconis2013}].

Since network data often contains sensitive information about individuals and their relationships (e.g., sexual relationships, email exchanges, financial transactions),
the data privacy has become an important issue in network data analysis.
The demand for privacy protection has leaded to a rapid development on algorithms to release network data or aggregate network statistics safely [e.g., \cite{Lu2014}, \cite{Task2012}].
\cite{dwork2006calibrating} developed a rigorous privacy standard--differential privacy to control the privacy leakage in the randomized data releasing mechanisms.
Roughly speaking, it says that changes to an individual data do not significantly affect the output distribution.

\cite{hay2009accurate} used the Laplace mechanism, which satisfies differential privacy, to release the degree partition of undirected graphs and
proposed an effective algorithm to find the minimum $L_2$ distance between all possible graph degree partitions and noisy degree partitions.
\cite{karwa2016inference} used the discrete Laplace mechanism to release the degree sequence and proved that a differentially private estimator of the parameter in the $\beta$--model is consistent and asymptotically normally distributed by using the denoised degree sequence under the assumption that all parameters are bounded.

\cite{pan2019asymptotics} showed that the moment estimators of the parameters directly based on the differentially private degree sequence without the denoised process is consistent and asymptotical normality in the $\beta$--model.
\cite{Yan2020} used the discrete Laplace mechanism to release the bi-degree sequence of directed graphs and
proved that the different differentially private estimator of the parameter in the $p_0$--model without the denoised process is consistent and asymptotically normally distributed, which is an exponential random graph model for directed networks with the bi-degree
sequence as the exclusively sufficient statistic.


In many realistic networks, edges could be weighted, taking a set of finite discrete values.
For example, \cite{freeman1979networkers} collected several social networks, whose directed edges between academics joining in an experiment on computer
mediated communication denote the acquaintance information that were coded as the discrete
values $0,\ldots, 4$. In this paper, we also use the discrete Laplace mechanism to release the bi-degree sequence of weighted directed graphs.
Motivated by \cite{Yan2020}, we use the moment equation to infer the degree parameters in the weighted $p_0$ model, in which the unobserved original degree sequence is directly replaced by the differentially private bi-degree sequence. The  moment estimator is differentially private.
We show that the differentially private moment estimator is consistent and asymptotically normal.
Numerical studies demonstrate our theoretical findings.

The article is organized in the following way. In Section \ref{section:model}, we give a brief introduction to the weighted $p_0$ model and some preliminaries of differential privacy. In Section \ref{section:results}, we present the differential private moment estimators and establish the asymptotic properties of the private estimator.
In Section \ref{section:simulation}, we carry out the numerical simulations. Then we make some discussion in Section \ref{section:discussion} to summarise out work. The proofs of the theorems are regelated into Section \ref{section:proof}.

\section{Model and differential privacy}
\label{section:model}

\subsection{The weighted $p_0$ model}

Let $G_n$ be a directed graph on $n\geq 2$ nodes labeled by ``1, \ldots, n" with no self-loops.
Let $a_{i,j}\in \Omega$ be the weight of the directed edge from node $i$ to node $j$,
where $\Omega\subset \mathbb{R}$ is the set of all possible weight values, and let $A=(a_{i,j})$ be the adjacency matrix of the graph $G_n$.
We consider a finite discrete weight here and assume that  $\Omega=\{0,1,\ldots,q-1\}$ with $q$ a fixed constant.
Since $G_n$ is loopless, let $a_{i,i}=0$ for convenience.
Let $d_i^+= \sum_{j \neq i}^{n} a_{i,j}$ be the out-degree of node $i$
and $d^+=(d_1^+, \ldots, d_n^+)^\top$ be the out-degree sequence of the graph $G_n$.
Similarly, define $d_i^- = \sum_{j \neq i}^{n} a_{j,i}$ as the in-degree of node $i$
and $d^-=(d_1^-, \ldots, d_n^-)^\top$ as the in-degree sequence.
The pair $\{d^+, d^-\}$ is called the bi-degree sequence.

The $p_0$ model [\cite{Yan2016asymptotics}] is an exponential random graph model for directed networks with the bi-degree
sequence as the exclusively sufficient statistic.
The density or probability mass function on $G_n$ in the weighted $p_0$ model [\cite{Zhang2016Directed}] can be represented as:
\begin{equation}
\label{p0model}
\P(G_n)= \frac{1}{c(\alpha, \beta)} \exp( \sum_i \alpha_i d_i^+ + \sum_j \beta_j d_j^- ),
\end{equation}
where $c(\alpha, \beta)$ is a normalizing constant, $\alpha=(\alpha_1, \ldots, \alpha_n)^\top$
and $\beta=(\beta_1, \ldots, \beta_n)^\top$.
The outgoingness parameter $\alpha_i$ characterizes how attractive the node is and the incomingness parameter $\beta_{i}$
illustrate the extent to which the node is attracted to others as in \cite{holland1981exponential}.
The subscript ``0" means a simpler model than the $p_1$ model
that contains an additional reciprocity parameter [\cite{holland1981exponential}].
\cite{Yan2016asymptotics} established the consistency and asymptotic normality of the MLE in the $p_0$ model in the case of binary weights and continuous weights.
\cite{Zhang2016Directed} extended their work to the case of finite discrete weights and derived the parallel results.

Since an out-edge from node $i$ pointing to $j$ is the in-edge of $j$ coming from $i$, it leads to that
the sum of out-degrees is equal to the sum of in-degrees.
If one transforms $(\alpha, \beta)$ to $(\alpha-c, \beta+c)$, then the probability distribution in \eqref{p0model} does not change.
For the sake of the identification of model parameters, we set $\beta_n=0$ as in \cite{Yan2016asymptotics}.
Moreover, the weighted $p_0$ model can be formulated by an array of mutually independent Bernoulli random variables $a_{i,j}$, $1\le i\neq j\le n$ with probabilities [\cite{Zhang2016Directed}]:
\[
\P(a_{i,j}=a)=\frac{e^{a(\alpha_{i}+\beta_{j})}}{\displaystyle\sum^{q-1}_{k=0}e^{k(\alpha_{i}+\beta_{j})}},~~a=0,1,\ldots,q-1.
\]
The normalizing constant $c(\alpha, \beta)$ is equal to $\sum_{i\neq j}^{n}\log(\sum_{k=0}^{q-1}e^{k(\alpha_{i}+\beta_{j})})$.
\subsection{Differential privacy}

Consider an original database $D$ containing a set of records of $n$ individuals.
A randomized data releasing mechanism $Q$ takes $D$ as an input and outputs a sanitized database $S=(S_1, \ldots, S_k)$
for public use, where the size of $S$ could not be the same as $D$.
Specifically, the mechanism $Q(\cdot|D)$ defines a conditional probability distribution on output $S$ given $D$.
Let $\epsilon$  be a positive real number and $\mathcal{S}$ denote the sample space of $Q$.
We call two databases $D_1$ and $D_2$ are neighbor if  they differ only on a single element.
The data releasing mechanism $Q$ is \emph{$\epsilon$-differentially private}
if for any two neighboring databases $D_1$ and $D_2$,
and all measurable subsets $B$ of $\mathcal{S}$ [\cite{dwork2006calibrating}],
\[
Q(S\in B |D_1) \leq e^{\epsilon }\times Q(S\in B |D_2).
\]

The definition of $\epsilon$-differential privacy is based on ratios of probabilities.
In particular, given two databases $D_1$ and $D_2$ that are different from only a single entry, the probability of an output $S$ given the input $D_1$ in the data releasing mechanism $Q$ is less than that given the input $D_2$
multiplied by a privacy factor $e^{\epsilon}$.
The privacy parameter $\epsilon$ is chosen by the data curator
administering the privacy policy and is public, which controls the trade-off
between privacy and utility.
Smaller value of $\epsilon$ means more privacy protection.
Notably, it is almost the same distribution of the output that an individual's record whether or not appears in the
database under preserve-privacy.

The concept of differential privacy depends on the definition of the two neighboring database.
In the graph field, \emph{differential privacy} is divided into \emph{node differential privacy} [\cite{kasiviswanathan2013analyzing}, \cite{Hay2010book}] and
\emph{edge differential privacy} [\cite{nissim2007smooth}].
Two graphs are called neighbors if they differ in exactly one
edge, then \emph{differential privacy} is \emph{edge differential privacy}.
Analogously, \emph{node differential privacy}  let graphs be neighbors if one can be obtained from the other by
removing a node and its adjacent edges.
Edge differential privacy protects edges not to be detected, whereas node differential privacy protects nodes together with their
adjacent edges, which is a stronger privacy policy.
Following \cite{hay2009accurate}, we use edge differential privacy here.
Let $\delta(G, G^\prime)$ be the number of edges
on which $G$ and $G^\prime$ differ.
The formal definition of edge differential privacy is as follows.

\begin{definition}[Edge differential privacy]
Let $\epsilon>0$ be a privacy parameter. Let $G_1$ and $G_2$ be arbitrarily two neighboring graphs that differ in exactly one
edge. A randomized mechanism
$Q(\cdot |G)$ is $\epsilon$-edge differentially private if
\[
\sup_{ G, G^\prime \in \mathcal{G}, \delta(G, G^\prime)=1 } \sup_{ S\in \mathcal{S}}  \log \frac{ Q(S|G) }{ Q(S|G^\prime ) } \le \epsilon,
\]
where $\mathcal{G}$ is the set of all directed graphs of interest on $n$ nodes and
$\mathcal{S}$ is the set of all possible outputs.
\end{definition}

Edge differential privacy requires that
the logarithmic ratio of the probabilities of an output $S$ given two neighboring graphs $G_1$ and $G_2$ is up to
at most a privacy scalar $\epsilon$.
If the outputs are the network statistics, then a simple algorithm to guarantee edge differential privacy is the Laplace Mechanism [e.g., \cite{dwork2006calibrating}]
that adds the Laplace noise.
When $f(G)$ is integer, one can use a discrete Laplace random variable as the noise as in \cite{karwa2016inference}, where it has the probability mass function:
\begin{equation}
\label{equ:discrete}
\P(X=x)= \frac{1-\lambda}{1+\lambda} \lambda^{|x|},~~x \in \{0, \pm 1, \ldots\}, \lambda\in(0,1).
\end{equation}

\begin{lemma}[Lemma 1 in \cite{karwa2016inference}]
\label{lemma:DLM}
Let $f:\mathcal{G} \to \R^k$. Let $e_1, \ldots, e_k$ be independent and identically distributed discrete Laplace random variables with
the parameter $\lambda$ in \eqref{equ:discrete}.
Then the discrete Laplace mechanism outputs $f(G)+(e_1, \ldots, e_k)$ is $\epsilon$-edge differentially private, where $\epsilon= -\Delta(f)\log \lambda$
and
\[
\Delta(f)=\max_{ \delta(G, G^\prime)=1 } \| f(G) - f(G^\prime) \|_1.
\]
\end{lemma}

One nice property of differential privacy is that any function of a differentially
private mechanism is also differentially private [\cite{dwork2006calibrating}].
That is, if $f$ is an
output of an $\epsilon$-differentially private mechanism, then
$g(f(G))$ is also $\epsilon$-differentially private, where $g$ is any function.
Therefore, any post-processing done on the differentially private bi-degree sequence  is also
differentially private.

\section{Main Results}
\label{section:results}

\subsection{A differentially private bi-degree sequence}

One common approach to provide privacy protection is using the Laplace mechanism, in which
independently and identically distributed Laplace random variables are added into the original data.
We use the discrete Laplace mechanism in Lemma \ref{lemma:DLM} to
release the bi-degree sequence $d=(d^+, d^-)$ under edge differential privacy.
Note that $f(G_n)=(d^+, d^-)$.
If we add a number $c$ to the weight $a_{ij}$ associated with the directed edge from $i$ to $j$,
then the out-degree of the head node $i$ increases $c$ and
the in-degree of the tail node $j$ decreases $c$. Similarly, if we subtract a number of $c$ from the weight $a_{ij}$, then
the out-degree of the head node $i$ decreases $c$ and
the in-degree of the tail node $j$ increases $c$.
Note that the largest changed number is $q-1$.
Therefore, the global sensitivity $\Delta(f)$ for the bi-degree sequence is $2(q-1)$.
We obtain the output  $z$ according to the discrete Laplace mechanism as follows:
\begin{equation}\label{equation:df}
\begin{array}{lcl}
z_i^+=d_i^+ + e_i^+,~~i=1, \ldots, n, \\
z_i^- = d_j^- + e_j^-,~~j=1, \ldots, n. \\
\end{array}
\end{equation}
where the random variables $\{ e_i^+ \}_{i=1}^n$ and $\{ e_i^- \}_{i=1}^n$ are independently generated from the discrete Laplace distribution in \eqref{equ:discrete} with the parameter $\lambda=e^{-\epsilon/(2(q-1))}$.

\subsection{Estimation in the weighted $p_0$ model}

\cite{Yan2020} used the moment equation to get the differential privacy estimation in the $p_0$ model for binary edges.
Motivated by his work, we also use the moment equation directly based on the differentially private bi-degree sequence with weighted edges here.
Formally, we use the following equations to estimate the degree parameters:
\renewcommand{\arraystretch}{1.2}
\begin{equation}\label{eq:likelihood-DP}
\large
\begin{array}{lcl}
z_i^+  & = & \sum^{n}_{j=1,j\neq i}\frac{\sum^{q-1}_{k=0}ke^{k({\alpha}_{i}+{\beta}_{j})}}{\sum^{q-1}_{k=0}e^{k({\alpha}_{i}+{\beta}_{j})}}, ~~i=1, \ldots, n, \\
z_j^-  & = & \sum^{n}_{i=1,i\neq j}\frac{\sum^{q-1}_{k=0}ke^{k({\alpha}_{i}+{\beta}_{j})}}{\sum^{q-1}_{k=0}e^{k({\alpha}_{i}+{\beta}_{j})}}, ~~j=1, \ldots, n-1,
\end{array}
\end{equation}
where $z$ is the differentially private bi-sequence in \ref{equation:df}.
We use the fixed point iteration algorithm to get the solutions of the above system equations.
Since $E(e_i^+)=0$ and $E(e_i^-)=0$, $i=1, \ldots, n$, the above equations are also the moment equations.
Let $\theta=(\alpha_1, \ldots, \alpha_n, \beta_1, \ldots, \beta_{n-1})^\top$.
The solution $\widehat{\theta}$ to the equations \eqref{eq:likelihood-DP} is the differentially private estimator of $\theta$,
where $\widehat{\theta}=(\hat{\alpha}_1, \ldots, \hat{\alpha}_n, \hat{\beta}_1, \ldots, \hat{\beta}_{n-1} )^\top$
and $\hat{\beta}_n=0$.

\subsection{Asymptotic properties of the estimator}
\label{section:asymptotic}
In this section, we present the consistency and asymptotical normality of the differentially private estimator.
For a vector $x=(x_1, \ldots, x_n)^\top\in R^n$, denote by
$\|x\|_\infty = \max_{1\le i\le n} |x_i|$, the $\ell_\infty$-norm of $x$.
The existence and consistency of $\widehat{\theta}$ is stated below.

\begin{theorem}\label{Theorem:con}
Assume that $A \sim \P_{ \theta^*}$, where $\P_{ \theta^*}$ denotes
the probability distribution \eqref{p0model} on $A$ under the parameter $\theta^*$.
If $\epsilon^{-1} e^{12\|\theta^*\|_\infty } = o( (n/\log n)^{1/2} )$,
 then with probability approaching one as $n$ goes to infinity, the estimator $\widehat{\theta}$ exists
and satisfies
\[
\|\widehat{\theta} - \theta^* \|_\infty = O_p\left( \frac{ (\log n)^{1/2} }{ n^{1/2} } (1+\kappa)e^{6\|\theta^*\|_\infty}  \right)=o_p(1).
\]
Further, if $\widehat{\theta}$ exists, it is unique.
\end{theorem}

\begin{remark}
The condition $\epsilon^{-1} e^{12\|\theta^*\|_\infty } = o( (n/\log n)^{1/2} )$ in Theorem \ref{Theorem:con} to guarantee the consistency of the estimator,
exhibits an interesting trade-off between the privacy parameter $\epsilon$ and $\|\theta^*\|_\infty$. If $\|\theta^*\|_\infty$ is bounded by a constant, $\epsilon$ can be as small as
$n^{1/2}/(\log n)^{-1/2}$. Conversely, if $e^{\|\theta^*\|_\infty}$ is growing at a rate of $n^{1/12}/(\log n)^{1/12}$,
then $\epsilon$ can only be at a constant magnitude.
\end{remark}

In order to present asymptotic normality of $\widehat{\theta}$, we introduce a class of matrices.
Given two positive numbers $m$ and $M$ with $M \ge m >0$, we say the $(2n-1)\times (2n-1)$ matrix $V=(v_{i,j})$ belongs to the class $\mathcal{L}_{n}(m, M)$ if the following holds:
\begin{equation}\label{eq:LmM}
\begin{array}{l}
m\le v_{i,i}-\sum_{j=n+1}^{2n-1} v_{i,j} \le M, ~~ i=1,\ldots, n-1; ~~~ v_{n,n}=\sum_{j=n+1}^{2n-1} v_{n,j}, \\
v_{i,j}=0, ~~ i,j=1,\ldots,n,~ i\neq j, \\
v_{i,j}=0, ~~ i,j=n+1, \ldots, 2n-1,~ i\neq j,\\
m\le v_{i,j}=v_{j,i} \le M, ~~ i=1,\ldots, n,~ j=n+1,\ldots, 2n-1,~ j\neq n+i, \\
v_{i,n+i}=v_{n+i,i}=0,~~ i=1,\ldots,n-1,\\
v_{i,i}= \sum_{k=1}^n v_{k,i}=\sum_{k=1}^n v_{i,k}, ~~ i=n+1, \ldots, 2n-1.
\end{array}
\end{equation}
Clearly, if $V\in \mathcal{L}_{n}(m, M)$, then $V$ is a $(2n-1)\times (2n-1)$ diagonally dominant, symmetric nonnegative
matrix.
Define $v_{2n,i}=v_{i,2n}:= v_{i,i}-\sum_{j=1;j\neq i}^{2n-1} v_{i,j}$ for $i=1,\ldots, 2n-1$ and $v_{2n,2n}=\sum_{i=1}^{2n-1} v_{2n,i}$.
\cite{Yan2016asymptotics} proposed to approximate the inverse of $V$, $V^{-1}$, by the matrix $S=(s_{i,j})$, which is defined as
\begin{equation}
\label{definition:S}
s_{i,j}=\left\{\begin{array}{ll}\frac{\delta_{i,j}}{v_{i,i}} + \frac{1}{v_{2n,2n}}, & i,j=1,\ldots,n, \\
-\frac{1}{v_{2n,2n}}, & i=1,\ldots, n,~~ j=n+1,\ldots,2n-1, \\
-\frac{1}{v_{2n,2n}}, & i=n+1,\ldots,2n-1,~~ j=1,\ldots,n, \\
\frac{\delta_{i,j}}{v_{i,i}}+\frac{1}{v_{2n,2n}}, & i,j=n+1,\ldots, 2n-1,
\end{array}
\right.
\end{equation}
where $\delta_{i,j}=1$ when $i=j$ and $\delta_{i,j}=0$ when $i\neq j$.

We use $V$ to denote the Fisher information matrix of $\theta$ in the weighted $p_0$ model.
It can be shown that for $i=1,\ldots,n$,
\[
v_{i,j}=0,j=1,\dots,n,j\neq i;
v_{i,i}=-\sum^{n}_{j=1,j\neq i}\frac{\sum^{}_{0\leq k<l\leq q-1}(k-l)^2 e^{(k+l)({\alpha}_{i}+ {\beta}_{j})}}{(\sum^{q-1}_{k=0}e^{k( {\alpha}_{i}+ {\beta}_{j})})^2},
\]
\[
v_{i,n+j}=-\frac{\sum^{}_{0\leq k<l\leq q-1}(k-l)^2 e^{(k+l)( {\alpha}_{i}+ {\beta}_{j})}}{(\sum^{q-1}_{k=0}e^{k( {\alpha}_{i}+ {\beta}_{j})})^2},j=1,\dots,n-1,j\neq i;
v_{i,n+i}=0,
\]
and for $j=1,\ldots,n-1,$
\[
v_{n+j,l}=-\frac{\sum^{}_{0\leq k<l\leq q-1}(k-l)^2 e^{(k+l)( {\alpha}_{i}+ {\beta}_{j})}}{(\sum^{q-1}_{k=0}e^{k( {\alpha}_{i}+ {\beta}_{j})})^2},l=1,\dots,n,l\neq j;
v_{n+j,j}=0,
\]
\[
v_{n+j,n+j}=-\sum^{n}_{i=1,i\neq j}\frac{\sum^{}_{0\leq k<l\leq q-1}(k-l)^2 e^{(k+l)( {\alpha}_{i}+ {\beta}_{j})}}{(\sum^{q-1}_{k=0}e^{k( {\alpha}_{i}+ {\beta}_{j})})^2};
v_{n+j,n+l}=0,l=1,\dots,n-1,l\neq j.
\]
\cite{Zhang2016Directed} show:
\[
\frac{n-1}{2(1+e^{2\| {\boldsymbol{\theta}}^*\|_{\infty}})}\leq v_{i,i}\leq \frac{(n-1)(q-1)^2}{2}, ~~ i=1, \ldots, 2n.
\]
Therefore $V\in \mathcal{L}_n(m,M)$, where $m$ is the left expression and $M$ is the right expression in the above inequality.
The asymptotic distribution of $\widehat{\theta}$ depends on $V$.
Since $V^{-1}$ does not have a closed form, we work with $S$ defined at \eqref{definition:S} to approximate it.
We formally state the central limit theorem as follows.

\begin{theorem}\label{Theorem:binary:central}
Assume that $A\sim \P_{\theta^*}$ and $(1+\frac{4(q-1)^2}{\epsilon_n})^2 e^{ 18\|\theta^*\|_\infty } = o( (n/\log n)^{1/2} )$.\\
(i) If $\frac{4(q-1)^2}{\epsilon_n} (\log n)^{1/2} e^{2\|\theta^*\|_\infty}=o(1)$ and $e^{2\|\theta^*\|_\infty}=o(n^{1/2})$,
then for any fixed $k\ge 1$, as $n \to\infty$, the vector consisting of the first $k$ elements of $(\widehat{\theta}-\theta^*)$ is asymptotically multivariate normal with mean $\mathbf{0}$ and covariance matrix given by the upper left $k \times k$ block of $S$ defined at \eqref{definition:S}.\\
(ii) Let
\[
s_n^2=\mathrm{Var}( \sum_{i=1}^n e_i^+ - \sum_{i=1}^{n-1} e_i^-)=(2n-1)(2e^{-\epsilon/2(q-1)})(1 - e^{-\epsilon/2(q-1)})^{-2}.
\]
If  $s_n/v_{2n,2n}^{1/2} \to c$ for some constant $c$,
then for any fixed $k \ge 1$, the vector consisting of the first $k$ elements of $(\widehat{\theta}-\theta^*)$ is asymptotically  $k$-dimensional multivariate normal distribution with mean $\mathbf{0}$
and covariance matrix
\[
\mathrm{diag}( \frac{1}{v_{1,1}}, \ldots, \frac{1}{v_{k,k}})+ (\frac{1}{v_{2n,2n}} + \frac{s_n^2}{v_{2n,2n}^2}) \mathbf{1}_k \mathbf{1}_k^\top,
\]
where $\mathbf{1}_k$ is a $k$-dimensional column vector with all entries $1$.
\end{theorem}

\section{Simulations}
\label{section:simulation}

The parameters in the simulations are as follows. Similar to \cite{Yan2016asymptotics}, the setting of the parameters ${\theta}^*$ took a linear form. Specifically, we set $\alpha_{i}^* = (i-1)L/(n-1)$ for $i=1,\ldots,n$. For simplicity, we set $\beta^*=\alpha^*$. We considered four different values for $L$, $L=0$, $\log(\log n)$,
$(\log n)^{1/2}$and $\log n $, respectively.
We simulated three different values for $\epsilon$: two are fixed ($\epsilon=3,2 $) and the other two values tend to zero with $n$, i.e., $\epsilon=
\log (n)/n^{1/4}, \log(n)/n^{1/2}$. We considered two values for $n$, $n=100$ and $200$. Each simulation was repeated $10,000$ times.

By Theorem \ref{Theorem:binary:central}, $\hat{\xi}_{i,j}=[\hat{\alpha}_{i}-\hat{\alpha}_{j}-(\alpha^{*}_i-\alpha^{*}_j)]/(1/\hat{v}_{i,i}+1/\hat{v}_{j,j})^{1/2}$,
$\hat{\zeta}_{i,j}=[\hat{\alpha}_{i}+\hat{\beta}_{j}-(\alpha^{*}_i-\beta^{*}_j)]/(1/\hat{v}_{i,i}+1/\hat{v}_{n+j,n+j})^{1/2}$, and
$\hat{\eta}_{i,j}=[\hat{\beta}_{i}-\hat{\beta}_{j}-(\beta^{*}_i-\beta^{*}_j)]/(1/\hat{v}_{n+i,n+i}+1/\hat{v}_{n+j,n+j})^{1/2}$
converge in distribution to the standard normal distributions, where $\hat{v}_{ii}$ is the estimate of $v_{ii}$ by replacing ${\theta}^*$ with $\hat{\theta}$. Therefore, we assess the asymptotic normality of $\hat{\xi}_{i,j}$, $\hat{\zeta}_{i,j}$, and $\hat{\eta}_{i,j}$ using the coverage probability. We choose three special pairs $(1,2),(n/2,n/2+1)$ and $(n-1,n)$ for $(i,j)$. We record the coverage probability of the $95\%$ confidence interval, the length of the confidence interval, and the frequency that the estimates do not exist. The results for $\hat{\xi}_{i,j}$, $\hat{\zeta}_{i,j}$, and $\hat{\eta}_{i,j}$ are similar, thus only the results of $\hat{\xi}_{i,j}$ are reported.

\begin{table}[!htp]
\centering
\caption{The reported values are the coverage frequency ($\times 100\%$) for $\alpha_i-\alpha_j$ for a pair $(i,j)$ / the length of the confidence interval / the frequency ($\times 100\%$) that the estimate did not exist. }
\label{Table:a}
\vskip5pt
\scriptsize
\begin{tabular}{cccccc}
\hline
$n$    & $(i,j)$ &  $L=0$ & $L=\log( \log n)$ & $L=(\log(n))^{1/2}$ & $L=\log(n)$ \\
\hline
\multicolumn{6}{c}{$\epsilon=3$}\\
\hline
100 & (1,2)        & 94.45/0.35/0   & 94.64/0.83/2.18  & 94.27/1.38/72.38  & NA/NA/100 \\
    & (50,51)      & 92.86/0.35/0   & 93.67/0.55/2.18  & 94.10/0.73/72.38  & NA/NA/100 \\
    & (99,100)     & 95.84/0.35/0   & 92.43/0.41/2.18 & 93.74/0.46/72.38   & NA/NA/100 \\
200 & (1,2)        & 94.79/0.25/0   & 94.76/0.64/0.05   & 94.82/1.07/27.14& NA/NA/100 \\
    & (100,101)    & 93.41/0.25/0   & 94.41/0.41/0.05   & 94.82/0.54/27.14& NA/NA/100 \\
    & (199,200)    & 93.59/0.25/0  & 94.29/0.29/0.05    & 94.63/0.33/27.14& NA/NA/100 \\
\hline
\multicolumn{6}{c}{$\epsilon=2$}\\
\hline
100 & (1,2)& 93.72/0.35/0& 94.04/0.85/11.93 & 93.64/1.39/94.30& NA/NA/100 \\
    & (50,51)& 93.79/0.35/0& 91.91/0.55/11.93 & 92.72/0.74/94.30& NA/NA/100 \\
    & (99,100)& 94.56/0.35/0& 89.82/0.41/11.93  & 93.33/0.46/94.30 & NA/NA/100 \\
200 & (1,2)    & 94.33/0.25/0  & 94.29/0.64/0.69 & 94.23/1.07/66.18& NA/NA/100 \\
   & (100,101) & 94.01/0.25/0  & 93.33/0.41/0.69 & 94.32/0.54/66.18& NA/NA/100 \\
   & (199,200) & 94.04/0.25/0  & 92.75/0.29/0.69 & 93.99/0.33/66.18& NA/NA/100 \\
\hline
\multicolumn{6}{c}{$\epsilon=\log n/n^{1/4}$}\\
\hline
100 & (1,2)    & 92.37/0.35/0& 92.44/0.87/34.19 & 92.38/1.32/99.53& NA/NA/100 \\
    & (50,51)  & 85.96/0.35/0  & 88.71/0.56/34.19 & 90.85/0.71/99.53& NA/NA/100 \\
    & (99,100) & 89.36/0.35/0  & 85.11/0.41/34.19 & 87.23/0.46/99.53& NA/NA/100 \\
200 & (1,2)    & 93.62/0.25/0  & 93.48/0.64/5.68 & 93.32/1.09/93.80& NA/NA/100 \\
   & (100,101) & 86.89/0.25/0  & 91.38/0.41/5.68 & 93.44/0.55/93.80& NA/NA/100 \\
   & (199,200) & 85.81/0.25/0  & 89.03/0.29/5.68 & 92.909/0.33/93.80& NA/NA/100 \\
   \hline
\end{tabular}
\end{table}

Table \ref{Table:a} reports the coverage frequencies of the $95\%$ confidence interval for $\alpha_i - \alpha_j$, the length of the confidence interval, and the frequency that the estimates do not exist. When $\epsilon=3$ and $2$ , the simulated coverage frequencies are close to the nominal level, the length of the confidence interval increases as $\tilde{L}$ increases and decreases as $n$ increases, while there are deviations when $\epsilon=\log n/n^{1/4}$.
When $\epsilon=\log n/n^{1/2}$, all estimates fail to exist.

\section{Discussion}
\label{section:discussion}

We have presented the consistency and asymptotic normality of the edge differentially private  estimator of the parameter in the weighted $p_0$ model
without the denoised process.
The result shows that the edge differentially private sequence can be directly used to draw statistical inference.
It is worth noting that the conditions imposed on $Q_n$ may not
be best possible. In particular, the condition guaranteeing the asymptotic
normality is stronger than that guaranteeing the consistency. The consistency requires $e^{Q_n} =
(n/ \log n)^{1/6}$, while the asymptotic normality requires $e^{Q_n} =
(n/ \log n)^{1/18}$. Simulation studies suggest that the conditions on $Q_n$
might be relaxed.  Note that the asymptotic behavior of the moment estimator depends not only on $Q_n$, but
also on the configuration of the parameters. We will investigate this in the future.

\section{Proofs}
\label{section:proof}

\subsection{Preliminaries}
\label{section:Preliminaries}
We present several results that we will use in this section.

\subsubsection{Concentration inequality for sub-exponential random variables}

A random variable $X$ is {\em sub-exponential} with parameter $\kappa > 0$ if [e.g.,\cite{Vershynin2012}]
\begin{equation*}
[\E|X|^p]^{1/p} \leq \kappa p \quad \text{ for all } p \geq 1.
\end{equation*}
Sub-exponential random variables satisfy the following concentration inequality.

\begin{theorem}[Corollary 5.17 in \citeauthor{Vershynin2012} (\citeyear{Vershynin2012})] \label{theorem:concen:inq}
Let $X_1, \dots, X_n$ be independent centered random variables, and suppose each $X_i$ is sub-exponential with parameter $\kappa$.
Then for every $\epsilon \geq 0$,
\begin{equation*}
\P\left( \left| \frac{1}{n} \sum_{i=1}^n X_i \right| \geq \epsilon \right) \leq 2\exp\left[-\gamma \, n \cdot \min\Big(\frac{\epsilon^2}{\kappa^2}, \: \frac{\epsilon}{\kappa} \Big) \right],
\end{equation*}
where $\gamma > 0$ is an absolute constant.
\end{theorem}

Note that if $X$ is a sub-exponential random variable with parameter $X$, then the centered random variable $X-\E[X]$ is also sub-exponential with parameter $2 \kappa $. This follows from the triangle inequality applied to the $p$-norm, followed by Jensen's inequality for $p \geq 1$:
\begin{equation*}
\begin{split}
\big[\E\big|X-\E[X]\big|^p\big]^{1/p}
&\leq [\E|X|^p]^{1/p} + \big|\E[X]\big|
\leq 2[\E|X|^p]^{1/p}.
\end{split}
\end{equation*}

\begin{lemma}\label{lemma:subexp:laplace}
Let $X$ be a discrete Laplace random variable with the probability distribution
\[
\P(X=x)= \frac{1-\lambda}{1+\lambda} \lambda^{|x|},~~x=0, \pm 1, \ldots, \lambda\in(0,1).
\]
Then $X$ is sub-exponential with parameter  $2( \log \frac{1}{\lambda} )^{-1}$.
\end{lemma}

\begin{proof}
Note that
\[
\E |X|^p = \frac{ 2(1-\lambda)}{1+\lambda} \sum_{x=0}^\infty \lambda^x x^p \le \frac{ 2(1-\lambda)}{1+\lambda} \int_0^\infty t^p e^{-t\log \frac{1}{\lambda}} dt
\le \frac{ 2(1-\lambda)}{1+\lambda}(\frac{1}{ \log \frac{1}{\lambda} })^{p+1} \Gamma(p).
\]
It follows that
\[
[\E |X|^p]^{1/p} < 2^{1/p} (\frac{1}{ \log \frac{1}{\lambda} })^{1+1/p} p < 2p \frac{1}{ \log \frac{1}{\lambda} }
\]
\end{proof}

\subsubsection{Convergence rate for the Newton iterative sequence}

Recall that the definition of $F(\theta)$ is
\renewcommand{\arraystretch}{1.2}
\begin{equation}\label{eq:F-DP}
\large
\begin{array}{lll}
F_i( \theta ) &  =  & z_i^{+}-\sum^{n}_{j=1,j\neq i}\frac{\sum^{q-1}_{k=0}ke^{k({\alpha}_{i}+{\beta}_{j})}}{\sum^{q-1}_{k=0}
e^{k({\alpha}_{i}+{\beta}_{j})}}, ~~~  i=1, \ldots, n, \\
F_{n+j}( \theta ) & = & z_j^{-}-\sum^{n}_{i=1,i\neq j}\frac{\sum^{q-1}_{k=0}k e^{k({\alpha}_{i}+{\beta}_{j})}}{\sum^{q-1}_{k=0}e^{k({\alpha}_{i}+{\beta}_{j})}},  ~~~  j=1, \ldots, n, \\
F( \theta ) & = & (F_1( \theta ), \ldots, F_{2n-1}( \theta ))^\top.
\end{array}
\end{equation}

For the ad hoc system of equations \eqref{eq:F-DP}, \cite{Yan2016asymptotics} established a geometric convergence of rate for
the Newton iterative sequence.

\begin{theorem}[Theorem 7 in \cite{Yan2016asymptotics}]\label{theorem:Newton:converg}
Define a system of equations:
\begin{eqnarray*}
F_i(\theta) = d_i - \sum\limits_{k=1, k\neq i}^n f(\alpha_i + \beta_k),~~i=1, \ldots, n,
\\
F_{n+j}(\theta) = b_j - \sum\limits_{k=1, k\neq j}^n f(\alpha_k + \beta_j),~~j=1, \ldots, n-1,
\\
F(\theta) = ( F_1(\theta), \ldots, F_n(\theta), F_{n+1}(\theta), \ldots, F_{2n-1}(\theta))^\top,
\end{eqnarray*}
where $f(\cdot)$ is a continuous function with the third derivative.  Let $D\subset \R^{2n-1}$ be a convex set and assume for any $\mathbf{x}, \mathbf{y}, \mathbf{v}\in D$, we have
\begin{eqnarray}
\label{Newton-condition-a}
\|[F'(\mathbf{x}) - F'(\mathbf{y})]\mathbf{v}\|_\infty \le K_1 \|\mathbf{x} - \mathbf{y}\|_\infty \|\mathbf{v}\|_\infty, \\
\label{Newton-condition-b}
\max_{i=1,\ldots,2n-1} \|F_i'(\mathbf{x}) - F_i'(\mathbf{y})\|_\infty \le K_2 \| \mathbf{x} - \mathbf{y}\|_\infty,
\end{eqnarray}
where $F'(\theta)$ is the Jacobin matrix of $F$ on $\theta$ and $F_i'(\theta)$ is the gradient function of $F_i$ on $\theta$.
Consider $\theta^{(0)}\in D$ with $\Omega(\theta^{(0)}, 2r) \subset D$,
where $r=\| [F'(\theta^{(0)})]^{-1}F(\theta^{(0)}) \|_\infty$.
For any $\theta\in \Omega(\theta^{(0)}, 2r)$, we assume
\begin{equation}\label{Newton-condition-c}
F'(\theta) \in\mathcal{ L}_n(m, M) \mbox{~~~or~~} -F'(\theta) \in\mathcal{ L}_n(m, M).
\end{equation}
For $k=1, 2, \ldots$, define the Newton iterates $\theta^{(k+1)} = \theta^{(k)} - [ F'(\theta^{(k)})]^{-1} F(\theta^{(k)})$.
Let
\begin{equation}\label{definition-Newton-rho}
\rho = \frac{ c_1(2n-1)M^2K_1}{ 2m^3n^2 } + \frac{ K_2 }{ (n-1)m}.
\end{equation}
If $\rho r < 1/2$, then $\theta^{(k)}\in \Omega(\theta^{(0)}, 2r)$, $k=1, 2, \ldots$, are well-defined and satisfy
\begin{equation}\label{Newton-convergence-rate}
\| \theta^{(k+1)} - \theta^{(0)} \|_\infty \le r/(1-\rho r).
\end{equation}
Further, $\lim_{k\to\infty} \theta^{(k)}$ exists and the limiting point is precisely the solution of $F(\theta)=0$
in the range of $\theta \in \Omega(\theta^{(0)}, 2r)$.
\end{theorem}

\subsubsection{Approximate inverse for the matrix $V$}

To quantify the accuracy of using $S$ to approximate $V$, we define the matrix maximum norm $\|\cdot\|$ for a general matrix $A=(a_{i,j})$ by $\|A\|:=\max_{i,j} |a_{i,j}|$.
The upper bound of the approximation error is given below.

\begin{proposition}[Proposition 1 in \cite{Yan2016asymptotics}] \label{pro:inverse:appro}
If $V\in \mathcal{L}_n(m, M)$ with $M/m=o(n)$, then for large enough $n$,
$$\| V^{-1}-S \| \le \frac{c_1M^2}{m^3(n-1)^2}.$$
where $c_1$ is a constant that does not depend on $M$, $m$ and $n$.
\end{proposition}

\subsection{Proofs for Theorem 1}
\label{subsection:proofcon}
We will use the Newton method to prove the consistency by applying Theorem \ref{theorem:Newton:converg} to obtain the geometrical convergence rate of the Newton iterative sequence.
To achieve it, we verify the conditions in Theorem \ref{theorem:Newton:converg}.
Let $F'(\theta)$ be the Jacobian matrix of $F$ defined at \eqref{eq:F-DP} on $\theta$ and $F_i'(\theta)$ is the gradient function of $F_i$ on $\theta$.
The first condition is the Lipchitz continuous property on $F'(\theta)$ and $F_i'(\theta)$.
Note that the Jacobian matrix of $F^\prime (\theta)$ does not
depend on $z$. By Lemma 4 in \cite{Zhang2016Directed}, we have that
\begin{eqnarray}
\label{Newton-condition-a}
\|[F'(x) - F'(y)]v\|_\infty \le K_1 \|x - y\|_\infty \|v\|_\infty, \\
\label{Newton-condition-b}
\max_{i=1,\ldots,2n-1} \|F_i'(x) - F_i'(y)\|_\infty \le K_2 \| x - y\|_\infty,
\end{eqnarray}
where $K_1=4(n-1)(q-1)^3$ and $K_2=2(n-1)(q-1)^3$.
This verifies the first condition. The second condition is that the upper bound of $\| F(\theta^*)\|$ is in the order of $(n\log n)^{1/2}$, stated in the below lemma.

\begin{lemma}
\label{Lemma-binary-2}
Let $\kappa=2(q-1)(-\log \lambda)^{-1}=4(q-1)^2/\epsilon$, where $\lambda\in(0,1)$. The following holds:
\begin{equation}\label{eq:con4}
\max\{ \max_i|z_i^+ - \E (d_i^+) |, \max_j | z_j^- - \E( d_j^-)| \} = O_p( \sqrt{n\log n } + \kappa\sqrt{\log n} ).
\end{equation}
\end{lemma}

\begin{proof}
Note that $\{e_i^+\}_{i=1}^n$ and $\{ e_i^- \}_{i=1}^n$ are independently discrete Laplace random variables and sub-exponential with the same parameter $\kappa$ by Lemma \ref{lemma:subexp:laplace}.
By the concentration inequality in Theorem \ref{theorem:concen:inq}, 
we have
\begin{equation}\label{eq:maxe}
\P( \max_{i=1, \ldots, n} |e_i^+| \ge 2\kappa \sqrt{\frac{\log n}{\gamma}} ) \le \sum_i \P( |e_i^+| \ge 2\kappa \sqrt{\frac{\log n}{\gamma}} ) \le n\times e^{-2\log n} = \frac{1}{n}
\end{equation}
and
\begin{equation}\label{eq:sum:error}
\P( |\sum_{i=1}^n e_i^+| \ge 2\kappa \sqrt{\frac{n\log n}{\gamma}} ) \le 2 \exp( - \frac{\gamma}{n} \times \frac{n\log n}{\gamma} ) = \frac{2}{n},
\end{equation}
where $\gamma$ is an absolute constant appearing in the concentration inequality.
In Lemma 3 in \cite{Yan2016asymptotics}, they show that with probability at least $1-4n/(n-1)^2$,
\begin{equation}\label{eq:con3}
\max\{ \max_i|d_i^+ - \E (d_i^+) |, \max_j | d_j^- - \E( d_j^-)| \} \le (q-1)\sqrt{(n-1)\log (n-1) }.
\end{equation}
So, with probability at least $1-4n/(n-1)^2-2/n$, we have
\[
\max_{i=1, \ldots, n} |z_i^+ - \E (d_i^+) | \le \max_i|d_i^+ - \E (d_i^+) | + \max_i |e_i^+| \le (q-1)\sqrt{n\log n } + 2\kappa \sqrt{\frac{\log n}{\gamma}} .
\]
Similarly, with probability at least $1-4n/(n-1)^2-2/n$, we have
\[
\max_{i=1, \ldots, n} |z_i^- - \E (d_i^-) | \le (q-1)\sqrt{n\log n } + 2\kappa  \sqrt{\frac{\log n}{\gamma}} .
\]
Let $A$ and $B$ be the events:
\begin{equation*}
\begin{array}{rcl}
A & = & \{ \max_{i=1, \ldots, n} |z_i^+ - \E (d_i^+) | \le (q-1)\sqrt{n\log n } + 2\kappa  \sqrt{\frac{\log n}{\gamma}} \}, \\
B & = & \{ \max_{i=1, \ldots, n} |z_i^- - \E (d_i^-) | \le (q-1)\sqrt{n\log n } + 2\kappa  \sqrt{\frac{\log n}{\gamma}} \}.
\end{array}
\end{equation*}
Consequently, as $n$ goes to infinity, we have
\begin{eqnarray*}
\P(A\bigcap B) \ge 1 - \P(A^c) - \P(B^c) \ge 1- 8n/(n-1)^2-4/n  \to 1.
\end{eqnarray*}
This completes the proof.
\end{proof}

It can be easily checked that $-F'(\theta)\in \mathcal{L}_n(m, M)$, \cite{Zhang2016Directed} where $M=(q-1)^2/2$ and $m= 1/2( 1+ e^{ 2\|\theta\|_\infty })^2$.
We are now ready to present the proof of Theorem 1. 

\begin{proof}[Proof of Theorem 1] 
Assume that equation \eqref{eq:con4} holds.
In the Newton iterates, we choose $\theta^*$ as the initial value $\theta^{(0)}$.
If $\theta \in \Omega(\theta^*, 2r)$, then $-F'(\theta)\in \mathcal{L}_n(m, M)$ with
\begin{equation}
\label{eq:Mm}
M=\frac{(q-1)^2}{2},~~m=\frac{ 1 }{ 2( 1 + e^{2(\|\theta^*\|_\infty+2r) } )^2 }.
\end{equation}
To apply Theorem \ref{theorem:Newton:converg}, we need to calculate $r$ and $\rho r$ in this theorem.
Let
\[
\tilde{F}_{2n}(\theta^*)= =\sum_{i=1}^{n}F_{i}(\theta^*)-\sum_{i=1}^{n-1}F_{n+i}(\theta^*)=d_n^{-}-\sum_{i=1}^{n-1}\frac{\displaystyle\sum\nolimits^{q-1}_{k=0}k e^{k({\alpha}_{i}+{\beta}_{n})}}{\displaystyle\sum\nolimits^{q-1}_{k=0}e^{k({\alpha}_{i}+{\beta}_{n})}}+\sum_{i=1}^{n}e_{i}^{+}-\sum_{i=1}^{n-1}e_{i}^{-}.
\]

By \eqref{eq:sum:error} and \eqref{eq:con3}, we have
\[
|\tilde{F}_{2n}(\theta^*)| = O_p( (1+\kappa )\sqrt{n\log n} ).
\]
By Proposition \ref{pro:inverse:appro}, we have
\begin{eqnarray*}
r& = &\parallel[F^{'} \theta^*]^{-1}F( \theta^*)\parallel_{\infty}\\
\vspace{0.3em}\\
& \leq & (2n-1)\|V^{-1} - S \| \| F(\theta^*) \|_\infty + \max_{i=1,\dots,2n-1}\frac{\mid F_{i}( \theta^*)\mid}{v_{i,i}}  +  \frac{\mid F_{2n}(\theta^*)\mid}{v_{2n,2n}}\\
\vspace{0.3em}\\
& \leq &\displaystyle\left(\frac{c_1(2n-1)\frac{1}{4}(q-1)^4}{\frac{1}{8(1+e^{2\parallel\theta^{*}\parallel_{\infty}})^3}(n-1)^2}
\right) \left( (q-1)\sqrt{(n-1)\log(n-1)}+\kappa\sqrt{\log n } \right)\\
\vspace{0.3em}\\
&&+\frac{2(1+e^{2\theta^*\parallel_{\infty}})}{n-1}  \left(  (1+\kappa)\sqrt{n\log n}+(q-1)\sqrt{(n-1)\log(n-1)}+\kappa\sqrt{\log n}  \right)\\
\vspace{0.3em}\\
& = & O_p \left( \frac{(\log n)^{1/2}}{n^{1/2}}(1+\kappa) e^{6\parallel\theta^{*}\parallel_{\infty}}  \right).
\end{eqnarray*}

Note that if $(1+\kappa) e^{6\|\theta^*\|_\infty } = o( (n/\log n)^{1/2} )$, then $r=o(1)$.
By \eqref{Newton-condition-a}, \eqref{Newton-condition-b} and \eqref{eq:Mm}, we have
\begin{eqnarray*}
\rho  =  \frac{ c_1(2n-1)M^24(n-1)(q-1)^3}{2m^3n^2} + \frac{ 2(n-1) }{ m(n-1)(q-1)^3 } = O( e^{6\|\theta^*\|_\infty} )
\end{eqnarray*}
Therefore, if $(1+\kappa) e^{12\|\theta^*\|_\infty } = o( (n/\log n)^{1/2} )$, then $\rho r\to 0$ as $n\to \infty$.
Consequently, by Theorem \ref{theorem:Newton:converg}, $\lim_{n\to \infty }\widehat{\theta}^{(n)}$ exists.
Denote the limiting point as $\widehat{\theta}$, then it satisfies
\begin{eqnarray*}
\|\widehat{\theta} - \theta^*\|_\infty  \le  2r = O\left(\frac{ (1+\kappa)(\log n)^{1/2}e^{6\|\theta^*\|_\infty} }{ n^{1/2} } \right) = o(1).
\end{eqnarray*}
By Lemma \ref{Lemma-binary-2}, equation \eqref{eq:con4} holds with probability approaching one such that
the above inequality also holds with probability approaching one.
The uniqueness of the solution to \eqref{eq:F-DP} is due to that $-F'(\theta)$ is positive definite.
\end{proof}

\subsection{Proofs for Theorem 2}  
\label{subsection:proofasymp}
The method of the proofs for the asymptotic normality of $\widehat{\theta}$ is similar to the method of the non-noisy case in \cite{Yan2016asymptotics}.
Wherein they work with the original bi-degree sequence $g$,
here we do with its noisy sequence $\tilde{g}$. The key step is to
represent $\hat{\theta} - \theta$ as the sum of $S(   \tilde{g} - \E g )$ and a remainder term.
For sake of clarity of exposition, we restate some results in \cite{Yan2016asymptotics} here.

\begin{lemma}[Lemma 8 \cite{Yan2016asymptotics}]\label{central:lemma1}
Let $R = V^{-1}-S$ and $U=\mbox{Cov}[R (g - \E g )]$. Then
\begin{equation}
\|U\|\le \|V^{-1}-S\|+\frac{6(q-1)^2(1+e^{2\|\bs{\theta}^*\|_\infty})^2}{(n-1)^2}.
\end{equation}
\end{lemma}

\begin{lemma}\label{central:lemma2}
Let $\kappa=2(q-1)(-\log \lambda)^{-1}=4(q-1)^2\epsilon^{-1}$.
If $(1+\kappa)^2 e^{ 18\|\theta^*\|_\infty } = o( (n/\log n)^{1/2} )$, then
for any $i$,
\begin{equation}
\widehat{\theta}_i- \theta_i^* = [V^{-1} (\tilde{g} - \E g ) ]_i + o_p( n^{-1/2} ).
\end{equation}
\end{lemma}

\begin{proof}
The proof is very similar to the proof of Lemma 9 in \cite{Yan2016asymptotics}.
It only requires verification of the fact that all the steps hold by replacing $d$ with $\tilde{d}$.
\end{proof}

The asymptotic normality of $\tilde{g}-\E g$ is stated in the following proposition, whose proof is in next section.

\begin{proposition}\label{pro:hatg}
Let $\kappa=2(q-1)(-\log \lambda)^{-1}$, where $\lambda=\exp(-\epsilon/2(q-1))$.
(i) If $\kappa_n (\log n)^{1/2} e^{2\|\theta^*\|_\infty}=o(1)$ and  $e^{\|\theta^*\|_\infty}=o( n^{1/2} )$, then for any fixed $k \ge 1$, as
$n\to\infty$, the vector consisting of the first $k$ elements of
$S(\tilde{g} - \E g )$ is asymptotically multivariate normal with mean zero
and covariance matrix given by the upper left $k \times k$ block of $S$.\\
(ii) Let
\[
s_n^2=\mathrm{Var}( \sum_{i=1}^n e_i^+ - \sum_{i=1}^{n-1} e_i^-) = (2n-1)\frac{ 2\lambda}{ (1-\lambda)^2}. 
\]
Assume that  $s_n/v_{2n,2n}^{1/2} \to c$ for some constant $c$. For any fixed $k \ge 1$, the vector consisting of the first $k$ elements of
$S(\tilde{g} - \E g )$ is asymptotically  $k$-dimensional multivariate normal distribution with mean $\mathbf{0}$ 
and covariance matrix
\[
\mathrm{diag}( \frac{1}{v_{1,1}}, \ldots, \frac{1}{v_{k,k}})+ (\frac{1}{v_{2n,2n}} + \frac{s_n^2}{v_{2n,2n}^2}) \mathbf{1}_k \mathbf{1}_k^\top,
\]
where $\mathbf{1}_k$ is a $k$-dimensional column vector with all entries $1$.
\end{proposition}

\begin{proof}[Proof of Theorem~2]  
By Lemma \ref{central:lemma2} and noting that $V^{-1}=S+R$, we have
\begin{equation*}
(\widehat{\theta}-\theta)_i= [S(\tilde{g} - \E g) ]_i+ [R \{ \tilde{g} - \E g \}]_i + o_p( n^{-1/2} ).
\end{equation*}
By \eqref{eq:maxe}, $\|\tilde{g} -  g \|_\infty = O_p( \kappa \sqrt{\log n})$.
So by proposition \ref{pro:inverse:appro}, we have
\[
[R ( \tilde{g} - g )]_i = O_p( n \frac{ M^2}{m^3n^2} \kappa \sqrt{\log n} ) = O_p( \frac{ \kappa (\log n)^{1/2}e^{6\| \theta ^*\|_\infty} }{n} ),
\]
where
\[
m=\frac{1}{2(1+e^{2\| \theta ^*\|_\infty})^2}, ~~M=\frac{(q-1)^2}{2}.
\]
If $\kappa e^{6\| \theta^*\|_\infty} = o( (n/\log n)^{1/2})$,  then $[R \{ z - d \}]_i=o_p(n^{-1/2})$.
Combing Lemma \ref{central:lemma1}, it yields
\[
[R ( \tilde{g} - \E g )]_i= [R(\tilde{g} -  g)]_i + [R(g - \E g)]_i = o_p(n^{-1/2}).
\]
Consequently,
\[
(\widehat{\theta}-\theta)_i = [S(\tilde{g} - \E g) ]_i + o_p( n^{-1/2} ).
\]
Theorem 2 
immediately follows from Proposition \ref{pro:hatg}.
\end{proof}

\section*{Acknowledgments}
We are grateful to the two anonymous referees for useful comments and suggestions. Wang's research is partially supported by the National Natural Science Foundation of China (No. 11771171). Luo¡¯s research is partially supported by National Natural Science Foundation of China(No.11801576) and by the Fundamental Research Funds for the Central Universities(South-Central University for Nationalities(CZQ19010)) and by National Statistical Science Research Project of China(No.2019LY59).

\bibliographystyle{apa}
\bibliography{p0reference}

\end{document}